\documentclass[a4]{article}

\usepackage{newtxtext}
\usepackage{amsmath,amsthm,amssymb,amscd}
\usepackage{tikz-cd}
\usepackage{hyperref}
\usepackage{authblk}

\newtheorem{thm}{Theorem}
\newtheorem{lem}[thm]{Lemma}
\newtheorem{prop}[thm]{Proposition}
\newtheorem{cor}[thm]{Corollary}

\theoremstyle{definition}
\newtheorem{dfn}[thm]{Definition}
\newtheorem{rem}[thm]{Remark}

\def\rank{\mathop{\mathrm{rank}}\nolimits}

\def\Hom{\mathop{\mathrm{Hom}}\nolimits}

\def\Ad{\mathop{\mathrm{Ad}}\nolimits}
\def\ad{\mathop{\mathrm{ad}}\nolimits}

\def\spn{\mathop{\mathrm{span}}\nolimits}

\def\SL{\mathop{\mathrm{SL}}\nolimits}

\newcommand{\mf}[1]{{\mathfrak{#1}}}
\newcommand{\bb}[1]{{\mathbb{#1}}}
\newcommand{\mca}[1]{{\mathcal{#1}}}

\newcommand{\df}{{d_{\mathcal{F}}}}
\newcommand{\GN}{{\Gamma\backslash N}}
\newcommand{\gn}{{\bar{\Gamma}\backslash\bar{N}}}
\newcommand{\er}{{e^{\bb{R}X_1}}}
\newcommand{\pz}{{\bar{\pi}^{-1}(z)}}

\title{Parameter rigid actions of simply connected nilpotent Lie groups}
\author{Hirokazu Maruhashi\thanks{maruhashihirokazu@gmail.com}}
\affil{Department of Mathematics, Kyoto University}
\date{\empty}

\begin{document}
\maketitle

\begin{abstract}
We show that for a locally free action of a simply connected nilpotent Lie group on a compact manifold, if every real valued cocycle is cohomologous to a constant cocycle, then the action is parameter rigid. The converse is true if the action has a dense orbit. Using this, we construct parameter rigid actions of simply connected nilpotent Lie groups whose Lie algebras admit rational structures with graduations. This generalizes the results of dos Santos \cite{dS} concerning the Heisenberg groups. 
\end{abstract}

\tableofcontents

\section{Introduction}
Let $G$ be a connected Lie group with Lie algebra $\mf{g}$ and $M$ a $C^\infty$-manifold without boundary. Let $\rho\colon M\times G\to M$ be a $C^\infty$ right action. We say $\rho$ is {\em locally free} if every isotropy subgroup of $\rho$ is discrete in $G$. For a locally free action $\rho$, we have the orbit foliation $\mca{F}$ of $\rho$, whose tangent bundle $T\mca{F}$ is naturally isomorphic to the trivial bundle $M\times\mf{g}$. 

A locally free action $\rho$ is {\em parameter rigid} if any $C^\infty$ right action $\rho^\prime$ of $G$ on $M$ with the same orbit foliation $\mca{F}$ is $C^\infty$-conjugate to $\rho$, more precisely, there exist an automorphism $\Phi$ of the Lie group $G$ and a $C^\infty$-diffeomorphism $F$ of $M$ which preserves each leaf of $\mca{F}$ and is $C^0$-homotopic to the identity map of $M$ through $C^\infty$-maps preserving each leaf of $\mca{F}$ such that 
\begin{equation*}
F\left(\rho(x,g)\right)=\rho^\prime\left(F(x),\Phi(g)\right)
\end{equation*}
holds for all $x\in M$ and $g\in G$. 

Parameter rigidity has been studied by several authors, for instance, Katok and Spatzier \cite{KS}, Matsumoto and Mitsumatsu \cite{MM}, Mieczkowski \cite{Mi}, dos Santos \cite{dS} and Ram\'{i}rez \cite{R}. Most of known examples of parameter rigid actions are those of abelian groups and actions of nonabelian Lie groups have not been considered so much. 

Parameter rigidity is closely related to cocycles over actions. Let $H$ be a Lie group. A $C^\infty$-map $c\colon M\times G\to H$ is called an {\em $H$-valued cocycle over $\rho$} if $c$ satisfies 
\begin{equation*}
c\left(x,gg^\prime\right)=c(x,g)c\left(\rho(x,g),g^\prime\right)
\end{equation*}
for all $x\in M$ and $g$, $g^\prime\in G$. 

A cocycle $c$ is {\em constant} if $c(x,g)$ is independent of $x$. A constant cocycle is just a homomorphism $G\to H$. 

Two $H$-valued cocycles $c$, $c^\prime$ are {\em cohomologous} if there exists a $C^\infty$-map $P\colon M\to H$ such that 
\begin{equation*}
c(x,g)=P(x)^{-1}c^\prime(x,g)P\left(\rho(x,g)\right)
\end{equation*}
holds for all $x\in M$ and $g\in G$. 

An action $\rho$ of $G$ on $M$ is {\em $H$-valued cocycle rigid} if every $H$-valued cocycle over $\rho$ is cohomologous to a constant cocycle. 

\begin{prop}[Matsumoto--Mitsumatsu \cite{MM}]\label{cp}
If a $C^\infty$ locally free action $\rho$ of a contractible Lie group $G$ on a closed $C^\infty$-manifold $M$ is $G$-valued cocycle rigid, then it is parameter rigid. 
\end{prop}

\begin{rem}
In \cite{MM} Matsumoto and Mitsumatsu assume that $\rho$ has at least one trivial isotropy subgroup, but this assumption is not necessary. 
\end{rem}

\begin{prop}[Matsumoto--Mitsumatsu \cite{MM}]
For a $C^\infty$ locally free action $\rho$ of $\bb{R}^n$ on a closed $C^\infty$-manifold $M$, the following are equivalent: 
\begin{enumerate}
\item The action $\rho$ is $\bb{R}$-valued cocycle rigid. 
\item The action $\rho$ is $\bb{R}^n$-valued cocycle rigid. 
\item The action $\rho$ is parameter rigid. 
\end{enumerate}
\end{prop}

\begin{rem}
The equivalence of the first two conditions is obvious. 
\end{rem}

In this paper we consider actions of simply connected nilpotent Lie groups. In \cite{dS}, dos Santos proved that for an action of a Heisenberg group $H_n$, the $\bb{R}$-valued cocycle rigidity implies the $H_n$-valued cocycle rigidity. Using this, he constructed parameter rigid actions of Heisenberg groups. To the best of my knowledge these are the only known ``nontrivial'' parameter rigid actions of nonabelian nilpotent Lie groups. We prove the following. 

\begin{thm}\label{main}
Let $N$ be a simply connected nilpotent Lie group, $M$ a closed $C^\infty$ manifold and $\rho$ a $C^\infty$ locally free action of $N$ on $M$. Then, 
\begin{enumerate}
\item the action $\rho$ is $\bb{R}$-valued cocycle rigid if and only if it is $N$-valued cocycle rigid. 
\item If $\rho$ is parameter rigid and has a dense orbit, then it is $\bb{R}$-valued cocycle rigid. 
\end{enumerate}
\end{thm}

Using this theorem we can construct parameter rigid actions of nilpotent Lie groups. The most interesting one is the following. 

\begin{thm}[Ram\'{i}rez \cite{R}]
Let $N\subset\SL(n,\bb{R})$ denote the group of all upper triangular matrices with ones on the diagonal, $\Gamma$ a cocompact lattice of ${\rm SL}(n,\bb{R})$ and $\rho$ the action of $N$ on $\Gamma\backslash{\rm SL}(n,\bb{R})$ by right multiplication. If $n\geq 4$, the action $\rho$ is $\bb{R}$-valued cocycle rigid. 
\end{thm}

\begin{rem}
In \cite{R}, Ram\'{i}rez proved more general theorems. 
\end{rem}

\begin{cor}
The above action $\rho$ is parameter rigid for $n\geq4$. 
\end{cor}

In Section \ref{const} we construct parameter rigid actions of nilpotent Lie groups using Theorem \ref{main}. It is a generalization of dos Santos' example. Let $N$ be a simply connected nilpotent Lie group and $\Gamma$, $\Lambda$ be lattices in $N$. Consider the action of $\Lambda$ on $\GN$ by right multiplication. Let $\tilde{\rho}$ be its suspended action of $N$ on $\Gamma\backslash N\times_\Lambda N$. 

\begin{thm}\label{main2}
If $\Lambda$ is Diophantine with respect to $\Gamma$, then the action $\tilde{\rho}$ of $N$ is parameter rigid. 
\end{thm}

For the definition of a Diophantine lattice, see Section \ref{const}. 

\section{Preliminaries}
Let $G$ be a contractible Lie group with Lie algebra $\mf{g}$, $M$ a closed $C^\infty$ manifold and $\rho$ a $C^\infty$ locally free action of $G$ on $M$ with orbit foliation $\mca{F}$. Let $H$ be a Lie group with Lie algebra $\mf{h}$. Let $\Omega^p(\mca{F};\mf{h})$ denote the set of all $C^\infty$-sections of $\Hom\left(\bigwedge^pT\mca{F},\mf{h}\right)$. The exterior derivative 
\begin{equation*}
\df\colon\Omega^p(\mca{F};\mf{h})\to\Omega^{p+1}(\mca{F};\mf{h})
\end{equation*}
is defined since $T\mca{F}$ is integrable. 

By differentiating, $H$-valued cocycles over $\rho$ are in one-to-one correspondence with $\mf{h}$-valued leafwise $1$-forms $\omega\in\Omega^1(\mca{F};\mf{h})$ such that 
\begin{equation*}
\df\omega+[\omega,\omega]=0. 
\end{equation*}

\begin{prop}\label{mat}
Let $c_1$, $c_2$ be $H$-valued cocycles over $\rho$ and let $\omega_1$, $\omega_2$ be the corresponding differential forms. For a $C^\infty$-map $P\colon M\to H$, the following are equivalent: 
\begin{enumerate}
\item $c_1(x,g)=P(x)^{-1}c_2(x,g)P(xg)$ for all $x\in M$ and $g\in G$. 
\item $\omega_1=\Ad\left(P^{-1}\right)\omega_2+P^*\theta$, where $\theta\in\Omega^1(H;\mf{h})$ is the left Maurer--Cartan form of $H$. 
\end{enumerate}
\end{prop}

\begin{cor}[Matsumoto--Mitsumatsu \cite{MM}]\label{comat}
The following are equivalent: 
\begin{enumerate}
\item The action $\rho$ is $G$-valued cocycle rigid. 
\item For each $\omega\in\Omega^1(\mca{F};\mf{g})$ such that $\df\omega+[\omega,\omega]=0$, there exist an endomorphism $\Phi\colon\mf{g}\to\mf{g}$ of Lie algebra and a $C^\infty$-map $P\colon M\to G$ such that 
\begin{equation*}
\omega=\Ad\left(P^{-1}\right)\Phi+P^*\theta. 
\end{equation*}
\end{enumerate}
\end{cor}

Proposition \ref{mat} is obtained by examining the proof of Corollary \ref{comat} in \cite{MM}. In this paper, we will identify a cocycle with its corresponding differential form. 

Let us consider real valued cocycles. A real valued cocycle over $\rho$ is given by $\omega\in\Omega^1(\mca{F};\bb{R})$ satisfying $\df\omega=0$. Two real valued cocycles $\omega_1$, $\omega_2$ are cohomologous if and only if $\omega_1=\omega_2+\df P$ for some $C^\infty$-function $P\colon M\to\bb{R}$. The leafwise cohomology $H^*(\mca{F})$ of $\mca{F}$ is the cohomology of the cochain complex $\left(\Omega^*(\mca{F};\bb{R}),\df\right)$. Thus $H^1(\mca{F})$ is the set of all equivalence classes of real valued cocycles. 

The identification $T\mca{F}\simeq M\times\mf{g}$ induces a map $H^*(\mf{g})\to H^*(\mca{F})$, where $H^*(\mf{g})$ is the cohomology of the Lie algebra $\mf{g}$. By the compactness of $M$, this map is injective on $H^1(\mf{g})$. Hence we identify $H^1(\mf{g})$ with its image. Note that $H^1(\mf{g})$ is the set of all equivalence classes of constant real valued cocycles. Thus real valued cocycle rigidity is equivalent to $H^1(\mca{F})=H^1(\mf{g})$. 

\section{Proof of Theorem \ref{main}}
Let $N$ be a simply connected nilpotent Lie group with Lie algebra $\mf{n}$, $M$ a closed $C^\infty$ manifold and $\rho$ a $C^\infty$ locally free action of $N$ on $M$ with orbit foliation $\mca{F}$. 

We first prove that $N$-valued cocycle rigidity implies real valued cocycle rigidity. There exist closed subgroups $N^\prime$ and $A$ of $N$ such that 
\begin{equation*}
N^\prime\triangleleft N,\quad N=N^\prime\rtimes A\ \ \text{and}\ \ A\simeq\bb{R}. 
\end{equation*}
Let $c$ be any real valued cocycle over $\rho$. We regard $c$ as an $N$-valued cocycle over $\rho$ via the inclusion $\bb{R}\simeq A\hookrightarrow N$. By the $N$-valued cocycle rigidity, there exist an endomorphism $\Phi$ of $N$ and a $C^\infty$-map $P\colon M\to N$ such that 
\begin{equation*}
c(x,g)=P(x)^{-1}\Phi(g)P\left(\rho(x,g)\right)
\end{equation*}
for all $x\in M$ and $g\in N$. Applying the natural projection $\pi\colon N\to A\simeq\bb{R}$, we obtain 
\begin{equation*}
c(x,g)=(\pi\circ P)(x)^{-1}(\pi\circ\Phi)(g)(\pi\circ P)\left(\rho(x,g)\right). 
\end{equation*}
Thus $c$ is cohomologous to a constant cocycle $\pi\circ\Phi$. 

Next we assume $H^1(\mca{F})=H^1(\mf{g})$ and prove the $N$-valued cocycle rigidity. We need the following two lemmas. 

\begin{lem}\label{lem1}
Let $V$ be a finite dimensional real vector space. Assume that $\omega\in \Omega^1(\mca{F};V)$ satisfies the equation $\df \omega=\varphi$, where $\varphi\in \Hom\left(\bigwedge^2\mf{n},V\right)$ is a constant leafwise $2$-form. Then there exists a constant leafwise $1$-form $\psi\in\Hom(\mf{n},V)$ with $\varphi=\df \psi$. 
\end{lem}

\begin{proof}
Since $N$ is nilpotent, there exists an $N$-invariant Borel probability measure $\mu$ on $M$. Define $\psi\in\Hom(\mf{n},V)$ by 
\begin{equation*}
\psi(X)=\int_M\omega(X)d\mu
\end{equation*}
for $X\in\mf{n}$. Since $\varphi(X,Y)=X\omega(Y)-Y\omega(X)-\omega\left([X,Y]\right)$ for all $X$, $Y\in\mf{n}$, we obtain 
\begin{equation*}
\varphi(X,Y)=-\int_M\omega\left([X,Y]\right)d\mu. 
\end{equation*}
Thus 
\begin{equation*}
\df\psi(X,Y)=-\psi\left([X,Y]\right)=-\int_M\omega\left([X,Y]\right)d\mu=\varphi(X,Y), 
\end{equation*}
hence $\df\psi=\varphi$. 
\end{proof}

Set $\mf{n}^1=\mf{n}$, $\mf{n}^i=\left[\mf{n},\mf{n}^{i-1}\right]$. Then $\mf{n}^s\neq 0$, $\mf{n}^{s+1}=0$ for some $s$. For each $1\leq i\leq s$, choose a subspace $V_i$ with $\mf{n}^i=V_i\oplus \mf{n}^{i+1}$, so that $\mf{n}=\bigoplus_{i=1}^s V_i$. 
 
\begin{lem}\label{induction}
Let $\omega\in\Omega^1(\mca{F};\mf{n})$ be such that $\df \omega+[\omega,\omega]=0$. Decompose $\omega$ as 
\begin{equation*}
\omega=\xi+\omega_k+\omega_{k+1}, 
\end{equation*}
where $\xi\in\Omega^1\left(\mca{F};\bigoplus_{i=1}^{k-1}V_i\right)$, $\omega_k\in\Omega^1\left(\mca{F};V_k\right)$ and $\omega_{k+1}\in\Omega^1\left(\mca{F};\mf{n}^{k+1}\right)$. If $\xi$ is constant, then there exists $\omega^\prime\in\Omega^1(\mca{F};\mf{n})$ with $\df\omega^\prime+\left[\omega^\prime,\omega^\prime\right]=0$ which is cohomologous to $\omega$ and such that 
\begin{equation*}
\omega^\prime=\xi^\prime+\omega_{k+1}^\prime, 
\end{equation*}
where $\xi^\prime\in\Omega^1\left(\mca{F};\bigoplus_{i=1}^kV_i\right)$ is constant and $\omega_{k+1}^\prime\in\Omega\left(\mca{F};\mf{n}^{k+1}\right)$. 
\end{lem}

\begin{proof}
By cocycle equation, 
\begin{equation*}
0=\df\xi+\df\omega_k+\df\omega_{k+1}+[\xi,\xi]+\text{an element of $\Omega^2\left(\mca{F};\mf{n}^{k+1}\right)$. }
\end{equation*}
Comparing the $V_k$-components of both sides of the equation, we see that $\df\omega_k$ is constant. Hence by Lemma \ref{lem1}, $\df\omega_k=\df\psi$ for some $\psi\in\Hom(\mf{n},V_k)$. Since we are assuming that $H^1(\mca{F})=H^1(\mf{n})$, there exist $\psi^\prime\in\Hom(\mf{n},V_k)$ and a $C^\infty$-map $h\colon M\to V_k$ such that 
\begin{equation*}
\omega_k=\psi+\psi^\prime+\df h. 
\end{equation*}
Put $P=e^h\colon M\to N$. Let $x\in M$ and $X\in T_x\mca{F}$. Choose a path $x(t)$ such that $X=\left.\frac{d}{dt}x(t)\right|_{t=0}$. Let $\theta\in\Omega^1(N;\mf{n})$ be the left Maurer--Cartan form of $N$. Then 
\begin{align*}
P^*\theta(X)&=\left.\frac{d}{dt}P(x)^{-1}P(x(t))\right|_{t=0}=\left.\frac{d}{dt}e^{-h(x)}e^{h(x(t))}\right|_{t=0}\\
&=\left.\frac{d}{dt}\exp\left(-h(x)+h(x(t))+\text{an element of $\mf{n}^{k+1}$}\right)\right|_{t=0}\\
&=\df h(X)+\text{an element of $\mf{n}^{k+1}$. }
\end{align*}
Thus $P^*\theta=\df h+\text{an element of $\Omega^1\left(\mca{F};\mf{n}^{k+1}\right)$}$. Note that $\Ad\left(P^{-1}\right)=\exp\ad(-h)$ is the identity on $\bigoplus_{i=1}^kV_i$ and preserves $\mf{n}^{k+1}$. Hence 
\begin{align*}
\omega-P^*\theta&=\xi+\psi+\psi^\prime+\text{an element of $\Omega^1\left(\mca{F};\mf{n}^{k+1}\right)$}\\
&=\Ad\left(P^{-1}\right)\left(\xi+\psi+\psi^\prime+\text{an element of $\Omega^1\left(\mca{F};\mf{n}^{k+1}\right)$}\right). 
\end{align*}
\end{proof}

Let $\omega$ be any $N$-valued cocycle. Using Lemma \ref{induction}, we can exchange $\omega$ for a cohomologous cocycle whose $V_1$-component is constant. Applying Lemma \ref{induction} repeatedly, we eventually get a constant cocycle cohomologous to $\omega$. This proves the $N$-valued cocycle rigidity. 

Next we assume that $\rho$ is parameter rigid and has a dense orbit. Let $\mf{n}^i$ and $V_i$ be as above. Note that $\mf{n}^s$ is central in $\mf{n}$. Fix a nonzero element $Z\in\mf{n}^s$. 

Let $[\omega]\in H^1(\mca{F})$. Let $\omega_0$ be the $N$-valued cocycle over $\rho$ corresponding to the constant cocycle $\mathrm{id}\colon N\to N$. We call $\omega_0$ {\em the canonical $1$-form of $\rho$}. Fix an $\epsilon>0$ and put $\eta:=\omega_0+\epsilon\omega Z$. Then $\eta$ is an $N$-valued cocycle over $\rho$ since 
\begin{equation*}
\df\eta+[\eta,\eta]=\df\omega_0+\epsilon(\df\omega)Z+\left[\omega_0,\omega_0\right]=0. 
\end{equation*}
Since $M$ is compact, we can assume $\eta_x\colon T_x\mca{F}\to\mf{n}$ is bijective for all $x\in M$ by choosing $\epsilon>0$ small. There exists a unique action $\rho^\prime$ of $N$ on $M$ whose orbit foliation is $\mca{F}$ and whose canonical $1$-form is $\eta$. See Asaoka \cite{A}. By the parameter rigidity, $\rho^\prime$ is conjugate to $\rho$. Thus there exist a $C^\infty$-map $P\colon M\to N$ and an automorphism $\Phi$ of $N$ satisfying 
\begin{equation}\label{star}
\omega_0+\epsilon\omega Z=\Ad\left(P^{-1}\right)\Phi_*\omega_0+P^*\theta.
\end{equation}
Note that $\log\colon N\to\mf{n}$ is defined since $N$ is simply connected and nilpotent. Let us decompose 
\begin{equation*}
\omega_0=\sum_{i=1}^s\omega_{0i},\quad \Phi_*\omega_0=\sum_{i=1}^s\omega_{0i}^\prime\ \ \text{and}\ \ \log P=\sum_{i=1}^sP_i
\end{equation*}
according to the decomposition $\mf{n}=\bigoplus_{i=1}^sV_i$. 

\begin{lem}\label{pp}
Assume that $P_1=\dotsm=P_{k-1}=0$, ie $\log P\in\mf{n}^k$. 
\begin{enumerate}
\item If $k<s$, then there exist a $C^\infty$-map $Q\colon M\to N$ and an automorphism $\Psi$ of $N$ such that 
\begin{equation*}
\omega_0+\epsilon\omega Z=\Ad\left(Q^{-1}\right)\Psi_*\omega_0+Q^*\theta
\end{equation*}
and $Q_1=\dotsm=Q_k=0$, where $\log Q=\sum_{i=1}^sQ_i$. 
\item If $k=s$, then $\omega$ is cohomologous to a constant cocycle. 
\end{enumerate}
\end{lem}

\begin{proof}
For all $X=\left.\frac{d}{dt}x(t)\right|_{t=0}\in T_x\mca{F}$, 
\begin{align*}
P^*\theta(X)&=\left.\frac{d}{dt}P(x)^{-1}P\left(x(t)\right)\right|_{t=0}=\left.\frac{d}{dt}\exp\left(-\sum_{i=k}^sP_i(x)\right)\exp\left(\sum_{i=k}^sP_i\left(x(t)\right)\right)\right|_{t=0}\\
&=\left.\frac{d}{dt}\exp\left\{\sum_{i=k}^s\left(P_i\left(x(t)\right)-P_i(x)\right)+\text{an element of $\mf{n}^{k+1}$}\right\}\right|_{t=0}\\
&=\left.\frac{d}{dt}\exp\left(P_k\left(x(t)\right)-P_k(x)+\text{an element of $\mf{n}^{k+1}$}\right)\right|_{t=0}\\
&=\df P_k(X)+\text{an element of $\mf{n}^{k+1}$. }
\end{align*}
We have 
\begin{align*}
\Ad\left(P^{-1}\right)\Phi_*\omega_0&=\exp\left(\ad\left(-\sum_{i=k}^sP_i\right)\right)\sum_{i=1}^s\omega_{0i}^\prime\\
&=\sum_{i=1}^s\omega_{0i}^\prime+\text{an element of $\mf{n}^{k+1}$. }
\end{align*}
Comparing the $V_k$-components of both sides of Equation \eqref{star} we get 
\begin{equation*}
\omega_{0k}+\delta_{ks}\epsilon\omega Z=\omega_{0k}^\prime+\df P_k. 
\end{equation*}
When $k=s$ the equation 
\begin{equation*}
\omega Z=\epsilon^{-1}\left(\omega_{0s}^\prime-\omega_{0s}\right)+\df\left(\epsilon^{-1}P_s\right)
\end{equation*}
shows that $\omega$ is cohomologous to a constant cocycle. 

If $k<s$, then $\df P_k=\phi\circ\omega_0$ for some linear map $\phi\colon\mf{n}\to V_k$. For any $X\in\mf{n}$, let $\tilde{X}$ denote the vector field on $M$ determined by $X$ via $\rho$. We have $\tilde{X}P_k=\phi(X)$ and by integrating over an integral curve $\gamma$ of $\tilde{X}$ we get 
\begin{equation*}
P_k\left(\gamma(T)\right)-P_k\left(\gamma(0)\right)=\phi(X)T
\end{equation*}
for all $T>0$. Since $M$ is compact, $\phi(X)=0$. Therefore $\df P_k=0$, so that $P_k$ is constant on each leaf of $\mca{F}$. Thus $P_k$ is constant on $M$ by our assumption. Put $g:=\exp(-P_k)$ and 
\begin{equation*}
Q:=gP=\exp\left(\sum_{i=k+1}^sP_i+\text{an element of $\mf{n}^{k+1}$}\right). 
\end{equation*}
Then 
\begin{align*}
\omega_0+\epsilon\omega Z&=\Ad\left(Q^{-1}g\right)\Phi_*\omega_0+\left(L_{g^{-1}}\circ Q\right)^*\theta\\
&=\Ad\left(Q^{-1}\right)\Psi_*\omega_0+Q^*\theta, 
\end{align*}
where $\Psi_*:=\Ad(g)\Phi_*$. 
\end{proof}

Applying Lemma \ref{pp} repeatedly, we see that $\omega$ is cohomologous to a constant cocycle. This completes the proof of Theorem \ref{main}. 

\section{Construction of parameter rigid actions}\label{const}
Let us now construct real valued cocycle rigid actions of nilpotent groups. For the structure theory of nilpotent Lie groups, see Corwin--Greenleaf \cite{CG}. 

Let $N$ be a simply connected nilpotent Lie group with Lie algebra $\mf{n}$. A basis $X_1,\dots,X_n$ of $\mf{n}$ is called a {\em strong Malcev basis} if $\spn_{\bb{R}}\left\{X_1,\dots,X_i\right\}$ is an ideal of $\mf{n}$ for each $i$. If $\Gamma$ is a lattice in $N$, there exists a strong Malcev basis $X_1,\dots,X_n$ of $\mf{n}$ such that $\Gamma=e^{\bb{Z}X_1}\cdots e^{\bb{Z}X_n}$. Such a basis is called a {\em strong Malcev basis strongly based on $\Gamma$}. 

\begin{dfn}
Let $\Gamma$ be a lattice in $N$. A lattice $\Lambda$ in $N$ is {\em Diophantine with respect to $\Gamma$} if there exists a strong Malcev basis $X_1,\dots,X_n$ of $\mf{n}$ strongly based on $\Gamma$ and a strong Malcev basis $Y_1,\dots,Y_n$ of $\mf{n}$ strongly based on $\Lambda$ such that $Y_i=\sum_{j=1}^ia_{ij}X_j$ for every $1\leq i\leq n$, where $a_{ii}$ is Diophantine. 
\end{dfn}

For lattices $\Gamma$ and $\Lambda$ of $N$, let $\rho$ be the action of $\Lambda$ on $\GN$ by right multiplication. First we will prove the following. 

\begin{thm}\label{constant}
If $\Lambda$ is Diophantine with respect to $\Gamma$, then every real valued $C^\infty$ cocycle $c\colon\GN\times\Lambda\to\bb{R}$ over $\rho$ is cohomologous to a constant cocycle. 
\end{thm}

\begin{proof}
Note that $X_1$ is in the center of $\mf{n}$. Let $\pi\colon N\to\bar{N}:=\er\backslash N$ be the projection. Since $\Gamma\cap\er=e^{\bb{Z}X_1}$ is a cocompact lattice in $\er$, the image $\bar{\Gamma}:=\pi(\Gamma)=\er\backslash\Gamma\er$ is a cocompact lattice in $\bar{N}$. Let $\bar{\mf{n}}=\bb{R}X_1\backslash\mf{n}$, then $\bar{X_2},\dots,\bar{X_n}$ is a strong Malcev basis of $\bar{\mf{n}}$ strongly based on $\bar{\Gamma}$. 

We will see that the naturally induced map $\bar{\pi}\colon\GN\to\gn$ is a principal $S^1$-bundle. Indeed, 
\begin{equation*}
\Gamma\backslash\Gamma\er\hookrightarrow\GN\twoheadrightarrow\Gamma\er\backslash N
\end{equation*}
is a principal $\Gamma\backslash\Gamma\er$-bundle and we have 
\begin{equation*}
\Gamma\backslash\Gamma\er\simeq\Gamma\cap\er\backslash\er=e^{\bb{Z}X_1}\backslash\er\simeq\bb{Z}\backslash\bb{R}
\end{equation*}
and 
\begin{equation*}
\begin{tikzcd}
\er\backslash\Gamma\er\ar[r,hook]&\er\backslash N\ar[r,two heads]\ar[d,two heads]&\Gamma\er\backslash N. \\
&\gn\ar[ru,"\sim",sloped]&
\end{tikzcd}
\end{equation*}
Since $\Lambda\cap\er=\Lambda\cap e^{\bb{R}Y_1}=e^{\bb{Z}Y_1}$ is a cocompact lattice in $e^{\bb{R}X_1}$, the image $\bar{\Lambda}:=\pi(\Lambda)$ is a cocompact lattice in $\bar{N}$. Then $\bar{Y_2},\dots,\bar{Y_n}$ is a strong Malcev basis of $\bar{\mf{n}}$ strongly based on $\bar{\Lambda}$ and $\bar{Y_i}=\sum_{j=2}^ia_{ij}\bar{X_j}$, where $a_{ii}$ is Diophantine. Therefore $\bar{\Lambda}$ is Diophantine with respect to $\bar{\Gamma}$. 

Since $\bar{\pi}$ is $\Lambda$-equivariant, the action $\rho$ of $\Lambda$ when restricted to $e^{\bb{Z}Y_1}$, preserves fibers of $\bar{\pi}$. 

Let $z\in\gn$. Choose a point $\Gamma x$ in $\pz$. Then we have a trivialization 
\begin{equation*}
\iota_{\Gamma x}\colon\bb{Z}\backslash\bb{R}\simeq\pz
\end{equation*}
of $\pz$ given by $\iota_{\Gamma x}(s)=\Gamma e^{sX_1}x$. Note that if we take another point $\Gamma y\in\pz$, $\iota_{\Gamma y}^{-1}\circ\iota_{\Gamma x}\colon\bb{Z}\backslash\bb{R}\to\bb{Z}\backslash\bb{R}$ is a rotation. 

Let $Y_1=aX_1$, where $a$ is Diophantine. If we identify $\pz$ with $\bb{Z}\backslash\bb{R}$ by $\iota_{\Gamma x}$, then the action of $e^{Y_1}$ on $\bb{Z}\backslash\bb{R}$ is $s\mapsto s+a$. 

Let $\mu_z$ be the normalized Haar measure naturally defined on $\bar{\pi}^{-1}(z)$, $\mu$ the $N$-invariant probability measure on $\GN$ and $\nu$ the $\bar{N}$-invariant probability measure on $\gn$. For any $f\in C(\GN)$, 
\begin{equation}\label{Weil}
\int_\GN fd\mu=\int_\gn\int_{\bar{\pi}^{-1}(z)}fd\mu_zd\nu. 
\end{equation}

\begin{lem}
The action $\rho$ is ergodic with respect to $\mu$.
\end{lem}

\begin{proof}
We use induction on $n$. For $n=1$, the action $\rho$ is given by an irrational rotation on $\bb{Z}\backslash\bb{R}$, hence the result is well known. In general, Let $f\colon\GN\to\bb{C}$ be a $\Lambda$-invariant $L^2$-function with $\int_\GN fd\mu=0$. Since the action of $e^{\bb{Z}Y_1}$ on $\pz$ is ergodic, $f|_{\pz}$ is constant $\mu_z$-almost everywhere. We denote this constant by $g(z)$. Then $g\colon\gn\to\bb{C}$ is $\bar{\Lambda}$-invariant measurable function. By induction, $g$ is constant $\nu$-almost everywhere. By Equation \eqref{Weil}, this constant must be zero. Therefore $f$ is zero $\mu$-almost everywhere. 
\end{proof}

Let $c\colon\GN\times\Lambda\to\bb{R}$ be a $C^\infty$-cocycle over $\rho$. We must show that $c$ is cohomologous to a constant cocycle $c_0\colon\Lambda\to\bb{R}$, where $c_0(\lambda):=\int_\GN c(x,\lambda)d\mu(x)$. Therefore we may assume that $\int_\GN c(x,\lambda)d\mu(x)=0$ for all $\lambda\in\Lambda$, and we will show that $c$ is a coboundary. We prove this by induction on $n$. When $n=1$, the action $\rho$ is given by a Diophantine rotation on $\bb{Z}\backslash\bb{R}$, hence the result is well known. 

\begin{lem}\label{int0}
For all $m\in\bb{Z}$, 
\begin{equation*}
\int_{\pz}c\left(s,e^{mY_1}\right)d\mu_z(s)=0. 
\end{equation*}
\end{lem}

\begin{proof}
Fix $m$ and put $g(z)=\int_{\pz}c\left(s,e^{mY_1}\right)d\mu_z(s)$. For any $\lambda\in\Lambda$, the cocycle equation gives 
\begin{equation*}
c(x,\lambda)+c\left(x\lambda,e^{mY_1}\right)=c\left(x,e^{mY_1}\right)+c\left(xe^{mY_1},\lambda\right). 
\end{equation*}
By integrating this equation on $\pz$, we get $g\left(z\pi(\lambda)\right)=g(z)$. Since the action of $\bar{\Lambda}$ on $\gn$ is ergodic, $g$ is constant. By Equation \eqref{Weil}, $g$ must be zero. 
\end{proof}

Let $f\colon\bb{Z}\backslash\bb{R}\xrightarrow{\iota_{\Gamma x}}\pz\xrightarrow{c(\ \cdot\ ,e^{Y_1})}\bb{R}$. We define $h_z\colon\pz\to\bb{R}$ by 
\begin{equation*}
h_z\left(\iota_{\Gamma x}(s)\right)=\sum_{k\in\bb{Z}\setminus\{0\}}\frac{\hat{f}(k)}{-1+e^{2\pi ika}}e^{2\pi iks}. 
\end{equation*}
Then $h_z\colon\pz\to\bb{R}$ is $C^\infty$, since $f$ is $C^\infty$ and $a$ is Diophantine. By Lemma \ref{int0}, we have 
\begin{equation*}
c\left(\iota_{\Gamma x}(s),e^{Y_1}\right)=-h_z\left(\iota_{\Gamma x}(s)\right)+h_z\left(\iota_{\Gamma x}e^{Y_1}\right). 
\end{equation*}
If we choose another point $\Gamma e^{s_0X_1}x\in\pz$ to define $h_z$, 
\begin{align*}
h_z\left(\iota_{\Gamma x}(s)\right)&=h_z\left(\Gamma e^{sX_1}x\right)=h_z\left(\iota_{\Gamma e^{s_0X_1}x}(s-s_0)\right)\\
&=\sum_{k\in\bb{Z}\setminus\{0\}}\frac{1}{-1+e^{2\pi ika}}\int_0^1c\left(\Gamma e^{(u+s_0)X_1}x,e^{Y_1}\right)e^{-2\pi iku}du\ e^{2\pi ik(s-s_0)}\\
&=\sum_{k\in\bb{Z}\setminus\{0\}}\frac{1}{-1+e^{2\pi ika}}\int_0^1f(u+s_0)e^{-2\pi iku}du\ e^{2\pi ik(s-s_0)}\\
&=\sum_{k\in\bb{Z}\setminus\{0\}}\frac{\hat{f}(k)}{-1+e^{2\pi ika}}e^{2\pi iks}, 
\end{align*}
so that $h_z$ is determined only by $z$. Define $h\colon\GN\to\bb{R}$ by $h|_{\pz}=h_z$. Then for all $x\in\GN$ and $m\in\bb{Z}$, we have $c\left(x,e^{mY_1}\right)=-h(x)+h\left(xe^{mY_1}\right)$. 

Let $U\subset\gn$ be an open set and $\sigma\colon U\to\bar{\pi}^{-1}(U)$ a section of $\bar{\pi}$. Then we have a trivialization $\bb{Z}\backslash\bb{R}\times U\simeq\bar{\pi}^{-1}(U)$ which sends $(s,z)$ to $\iota_{\sigma(z)}(s)=\Gamma e^{sX_1}\sigma(z)$. Hence 
\begin{equation*}
h\left(\iota_{\sigma(z)}(s)\right)=\sum_{k\in\bb{Z}\setminus\{0\}}\frac{1}{-1+e^{2\pi ika}}\int_0^1c\left(\iota_{\sigma(z)}(u),e^{Y_1}\right)e^{-2\pi iku}du\ e^{2\pi iks}
\end{equation*}
on $\bar{\pi}^{-1}(U)$. The following lemma shows $h$ is $C^\infty$ on $\GN$. 

\begin{lem}
Let $U\subset\bb{R}^n$ be an open set and $f\colon\bb{Z}\backslash\bb{R}\times U\to\bb{R}$ be a $C^\infty$-function. Define $h\colon\bb{Z}\backslash\bb{R}\times U\to\bb{R}$ by 
\begin{equation*}
h(s,z)=\sum_{k\in\bb{Z}\setminus\{0\}}\frac{1}{-1+e^{2\pi ika}}\widehat{f_z}(k)e^{2\pi iks}, 
\end{equation*}
where $f_z(u)=f(u,z)$. Then $h$ is $C^\infty$. 
\end{lem}

\begin{proof}
Let $V\subset\bb{R}^n$ be an open set such that $\bar{V}\subset U$ and $\bar{V}$ is compact. We will show that $h$ is $C^\infty$ on $\bb{Z}\backslash\bb{R}\times V$. Choose constants $C$, $\alpha>0$ such that $\left|-1+e^{2\pi ika}\right|\geq C\lvert k\rvert^{-\alpha}$ for all $k\in\bb{Z}\setminus\{0\}$. 

We will first prove that $h$ is continuous. Since for any $m\in\bb{Z}_{>0}$, 
\begin{equation*}
\frac{\partial^mf_z}{\partial s^m}(s)=\sum_{k\in\bb{Z}}(2\pi ik)^m\widehat{f_z}(k)e^{2\pi iks}
\end{equation*}
in $L^2(\bb{Z}\backslash\bb{R})$, 
\begin{align*}
\left\lVert\frac{\partial^mf_z}{\partial s^m}\right\rVert_2^2&=\sum_{k\in\bb{Z}}\left|(2\pi ik)^m\widehat{f_z}(k)\right|^2\\
&\geq(2\pi)^{2m}\lvert k\rvert^{2m}\left|\widehat{f_z}(k)\right|^2\geq\lvert k\rvert^{2m}\left|\widehat{f_z}(k)\right|^2. 
\end{align*}
Since $\left\lVert\frac{\partial^mf_z}{\partial s^m}\right\rVert_2=\left(\int_0^1\left|\frac{\partial^m}{\partial s^m}f(s,z)\right|^2ds\right)^{\frac{1}{2}}$ is continuous in $z$, there exists $M>0$ such that $\left\lVert\frac{\partial^mf_z}{\partial s^m}\right\rVert_2<M$ for every $z\in\bar{V}$. Hence for all $k\in\bb{Z}$ and $z\in\bar{V}$, $\lvert k\rvert^m\left|\widehat{f_z}(k)\right|\leq M$. Therefore, for any $z\in\bar{V}$, 
\begin{align*}
\sum_{k\in\bb{Z}\setminus\{0\}}\left|\frac{1}{-1+e^{2\pi ika}}\widehat{f_z}(k)e^{2\pi iks}\right|&\leq C^{-1}\sum_{k\in\bb{Z}\setminus\{0\}}\frac{1}{\lvert k\rvert^2}\lvert k\rvert^{\alpha+2}\left|\widehat{f_z}(k)\right|\\
&\leq C^{-1}M\sum_{k\in\bb{Z}\setminus\{0\}}\frac{1}{\lvert k\rvert^2}<\infty. 
\end{align*}
This implies the continuity of $h$ on $\bb{Z}\backslash\bb{R}\times\bar{V}$. 

We have 
\begin{equation*}
\frac{\partial h}{\partial s}(s,z)=\sum_{k\in\bb{Z}\setminus\{0\}}\frac{2\pi ik}{-1+e^{2\pi ika}}\widehat{f_z}(k)e^{2\pi iks}. 
\end{equation*}
Thus a similar argument shows that $\frac{\partial h}{\partial s}$ is continuous. 

Let $z=(z_1,\dots,z_n)$. For any $z\in\bar{V}$, 
\begin{align*}
\left|\frac{\partial}{\partial z_j}\left(\frac{1}{-1+e^{2\pi ika}}\widehat{f_z}(k)e^{2\pi iks}\right)\right|&=\left|\frac{1}{-1+e^{2\pi ika}}\widehat{\frac{\partial f}{\partial z_j}(\ \cdot\ ,z)}(k)e^{2\pi iks}\right|\\
&\leq C^{-1}\frac{1}{\lvert k\rvert^2}\lvert k\rvert^{\alpha+2}\left|\widehat{\frac{\partial f}{\partial z_j}(\ \cdot\ ,z)}(k)\right|\\
&\leq C^{-1}M^\prime\frac{1}{\lvert k\rvert^2}\ \in L^1\left(\bb{Z}\setminus\{0\}\right). 
\end{align*}
Thus 
\begin{equation*}
\frac{\partial h}{\partial z_j}(s,z)=\sum_{k\in\bb{Z}\setminus\{0\}}\frac{1}{-1+e^{2\pi ika}}\widehat{\frac{\partial f}{\partial z_j}(\ \cdot\ ,z)}(k)e^{2\pi iks}. 
\end{equation*}
Hence $\frac{\partial h}{\partial z_j}$ is continuous by an argument similar to those above. For higher derivatives of $h$, repeat this procedure. 
\end{proof}

Set 
\begin{equation*}
c_1(x,\lambda)=c(x,\lambda)+h(x)-h(x\lambda). 
\end{equation*}
The map $c_1\colon\GN\times\Lambda\to\bb{R}$ is a $C^\infty$-cocycle and $c_1\left(x,e^{mY_1}\right)=0$. Thus for any $\lambda\in\Lambda$, the cocycle equation implies $c_1(x,\lambda)=c_1\left(xe^{Y_1},\lambda\right)$. Since the action of $e^{\bb{Z}Y_1}$ on $\pz$ is ergodic, $c_1(x,\lambda)$ is constant on $\pz$. Therefore we can define a cocycle $\bar{c}\colon\gn\times\bar{\Lambda}\to\bb{R}$ by $\bar{c}\left(\bar{\pi}(x),\pi(\lambda)\right)=c_1(x,\lambda)$. Indeed, if $\bar{\pi}(x)=\bar{\pi}(y)$ and $\pi(\lambda)=\pi(\lambda^\prime)$, then there exists $m\in\bb{Z}$ with $\lambda=e^{mY_1}\lambda^\prime$, so that 
\begin{equation*}
c_1(x,\lambda)=c_1\left(x,e^{mY_1}\lambda^\prime\right)
=c_1\left(xe^{mY_1},\lambda^\prime\right)=c_1\left(y,\lambda^\prime\right). 
\end{equation*}
Furthermore, 
\begin{align*}
\int_\gn\bar{c}\left(x,\pi(\lambda)\right)d\nu(z)
&=\int_\gn\int_{\pz}c_1(s,\lambda)d\mu_z(s)d\nu(z) \\
&=\int_\GN c_1(x,\lambda)d\mu(x)=0. 
\end{align*}
By induction, there exists a $C^\infty$-function $P\colon\gn\to\bb{R}$ such that $\bar{c}\left(z,\pi(\lambda)\right)=-P(z)+P\left(z\pi(\lambda)\right)$. Put $Q=P\circ\bar{\pi}$. Then 
\begin{equation*}
c_1(x,\lambda)=\bar{c}\left(\bar{\pi}(x),\pi(\lambda)\right)=-Q(x)+Q(x\lambda). 
\end{equation*}
This proves Theorem \ref{constant}. 
\end{proof}

\begin{proof}[Proof of Theorem \ref{main2}]
Let $\tilde{\rho}\colon M\times N\to M$ be the suspension of $\rho\colon\GN\times\Lambda\to\GN$, where $M=\GN\times_{\Lambda}N$ is a compact manifold. Then $\tilde{\rho}$ is locally free. Let $\mca{F}$ be the orbit foliation of $\tilde{\rho}$. We have 
\begin{equation*}
H^1(\mca{F})\simeq H^1\left(\Lambda;C^\infty(\GN)\right)
\end{equation*}
by Pereira--dos Santos \cite{PdS}, where the right hand side is the first cohomology of the $\Lambda$-module $C^\infty(\GN)$ obtained by $\rho$. It is easy to prove that $\Hom(\Lambda,\bb{R})\to H^1\left(\Lambda;C^\infty(\GN)\right)$ is injective. By Theorem \ref{constant}, 
\begin{equation*}
H^1\left(\Lambda;C^\infty(\GN)\right)=\Hom(\Lambda,\bb{R}). 
\end{equation*}

\begin{lem}
\begin{equation*}
\dim\Hom(\Lambda,\bb{R})=\dim H^1(\mf{n}). 
\end{equation*}
\end{lem}

\begin{proof}
Recall that $[N,N]\backslash\Lambda[N,N]$ is a cocompact lattice in $[N,N]\backslash N$ and that $[\Lambda,\Lambda]\backslash\left(\Lambda\cap[N,N]\right)$ is finite. Since 
\begin{equation*}
0\to[\Lambda,\Lambda]\backslash\left(\Lambda\cap[N,N]\right)\to[\Lambda,\Lambda]\backslash\Lambda\to[N,N]\backslash\Lambda[N,N]\to0
\end{equation*}
is exact, we have 
\begin{equation*}
\rank\left([\Lambda,\Lambda]\backslash\Lambda\right)=\rank\left([N,N]\backslash\Lambda[N,N]\right)=\dim\left([N,N]\backslash N\right). 
\end{equation*}
Thus 
\begin{align*}
\dim\Hom(\Lambda,\bb{R})&=\dim\Hom\left([\Lambda,\Lambda]\backslash\Lambda,\bb{R}\right)\\
&=\rank\left([\Lambda,\Lambda]\backslash\Lambda\right)\\
&=\dim\left([N,N]\backslash N\right)\\
&=\dim\Hom_{\bb{R}}\left([\mf{n},\mf{n}]\backslash\mf{n},\bb{R}\right)\\
&=\dim H^1(\mf{n}). 
\end{align*}
\end{proof}

Therefore we obtain 
\begin{equation*}
H^1(\mca{F})=H^1(\mf{n}). 
\end{equation*}
This proves Theorem \ref{main2}. 
\end{proof}

\section{Existence of Diophantine lattices}
Let $N$ be a simply connected nilpotent Lie group with Lie algebra $\mf{n}$ and $\mf{n}_{\bb{Q}}$ a rational structure of $\mf{n}$. We will construct a Diophantine lattice when $\mf{n}_{\bb{Q}}$ admits a graduation. Namely, we assume that $\mf{n}_{\bb{Q}}$ has a sequence $V_i$ of $\bb{Q}$-subspaces such that $\mf{n}_{\bb{Q}}=\bigoplus_{i=1}^kV_i$ and $[V_i,V_j]\subset V_{i+j}$. Let $X_1,\dots,X_n$ be a $\bb{Q}$-basis of $\mf{n}_{\bb{Q}}$ such that 
\begin{equation*}
X_1,\dots,X_{i_1}\in V_k,\quad X_{i_1+1},\dots,X_{i_2}\in V_{k-1},\dots,X_{i_{k-1}+1},\dots,X_n\in V_1. 
\end{equation*}
Then $X_1,\dots,X_n$ is a strong Malcev basis of $\mf{n}$ with rational structure constants. Multiplying $X_1,\dots,X_n$ by an integer if necessary, we may assume that $\Gamma:=e^{\bb{Z}X_1}\cdots e^{\bb{Z}X_n}$ is a cocompact lattice in $N$. Let $\alpha$ be a root of an irreducible polynomial of degree $k+1$ over $\bb{Q}$. Since $\alpha,\alpha^2,\dots,\alpha^k$ are irrational algebraic numbers, they are Diophantine. If we define a linear map $\varphi\colon\mf{n}\to\mf{n}$ by $\varphi(X)=\alpha^iX$ for $X\in V_i\otimes\bb{R}$, then $\varphi$ is an automorphism of the Lie algebra $\mf{n}$. Put $Y_i=\varphi(X_i)$. Then $Y_1,\dots,Y_n$ is a strong Malcev basis of $\mf{n}$ strongly based on $\Lambda:=e^{\bb{Z}Y_1}\cdots e^{\bb{Z}Y_n}$. Thus $\Lambda$ is Diophantine with respect to $\Gamma$. 

\section*{Acknowledgement}
The author would like to thank Masayuki Asaoka for helpful comments.

\end{document}